\documentclass{article}

\usepackage{amssymb,amsmath,graphicx,ucs,hyperref}
\usepackage[utf8x]{inputenc}

\newtheorem{theorem}{Theorem}
\newtheorem{lemma}{Lemma}

\newenvironment{proof}{\noindent \emph{Proof. }}{\hfill \hbox{\rlap{$\sqcap$}$\sqcup$}\\}

\title{Shield Tilings}

\author{Thomas Fernique
\and
Olga Mikhailovna Sizova
}
\date{}

\begin{document}
\maketitle

\begin{abstract}
We provide a complete description of the edge-to-edge tilings with a regular triangle and a shield-shaped hexagon with no right angle.
The case of a hexagon with a right angle is also briefly discussed.
\end{abstract}

\section{Introduction}

Given a finite set of polygons called {\em tiles}, a {\em tiling} is a covering of the Euclidean plane by interior disjoint isometric copies of these polygons, with the property that the intersection of two polygons, if not empty, is either a vertex or an entire edge (so-called {\em edge-to-edge} condition).

A common issue, in particular in statistical mechanics, is to describe all the possible tilings for a given set of tiles.
This problem has been extensively studied in the case of $2\times 1$ rectangles called {\em dominoes} (see e.g. \cite{CKP01}) as well as in the case of a square and a triangle (see, e.g. \cite{ICJKS21,OH93}).

Here we focus on {\em shield tilings}, that are tilings by a unit regular triangle and a {\em shield}, defined as a hexagon with unit edges whose angles take two values $\alpha\in(\tfrac{\pi}{3},\tfrac{2\pi}{3})$ and $\beta=\tfrac{4\pi}{3}-\alpha$ that alternate when we go through the angles circularly (Fig.~\ref{fig:shield_tile}).

\begin{figure}[hbt]
\centering
\includegraphics[width=0.8\textwidth]{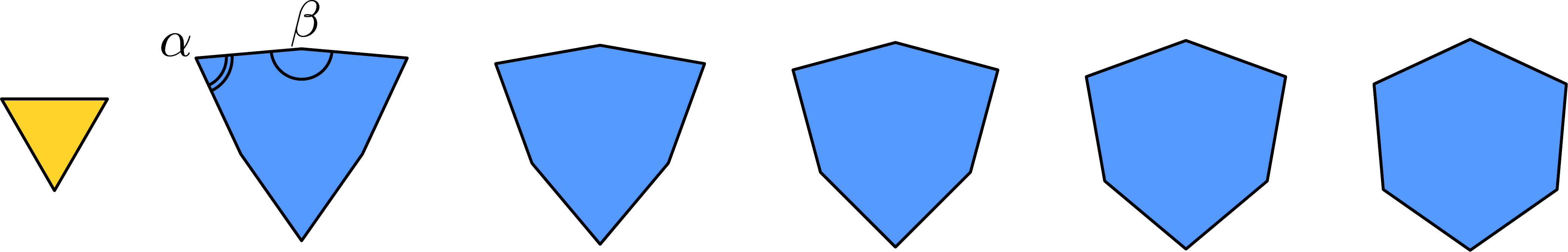}
\caption{The triangle tile and shield tiles for different values of $\alpha$.}
\label{fig:shield_tile}
\end{figure}

The motivation comes from a problem of classification of {\em disk packings}, that are a sets of interior-disjoint disks in the Euclidean plane.
More precisely, a disk packing is said to be {\em triangulated} if its {\em contact graph}, that is, the graph which connects the centers of adjacent disks, is a triangulation.
In \cite{Ken06}, it was proven that there are only $9$ values $r<1$ that allow a triangulated packing of disks of size $1$ and $r$, and a description of the possible packings was also provided, except in the case $r\approx 0.54$ root of
$$
r^8 - 8r^7 - 44r^6 - 232r^5 - 482r^4 - 24r^3 + 388r^2 - 120r + 9.
$$
In this case, the possible packings turned out to correspond to shield tilings for $\alpha\approx 99.34^\circ$ (Fig.~\ref{fig:shield_packing}) and only two examples of packings were provided.
Further similar cases with various values of $\alpha$ appear in the classification of triangulated packings by three sizes of disks \cite{FHS21}.

\begin{figure}[hbt]
\centering
\includegraphics[width=0.4\textwidth]{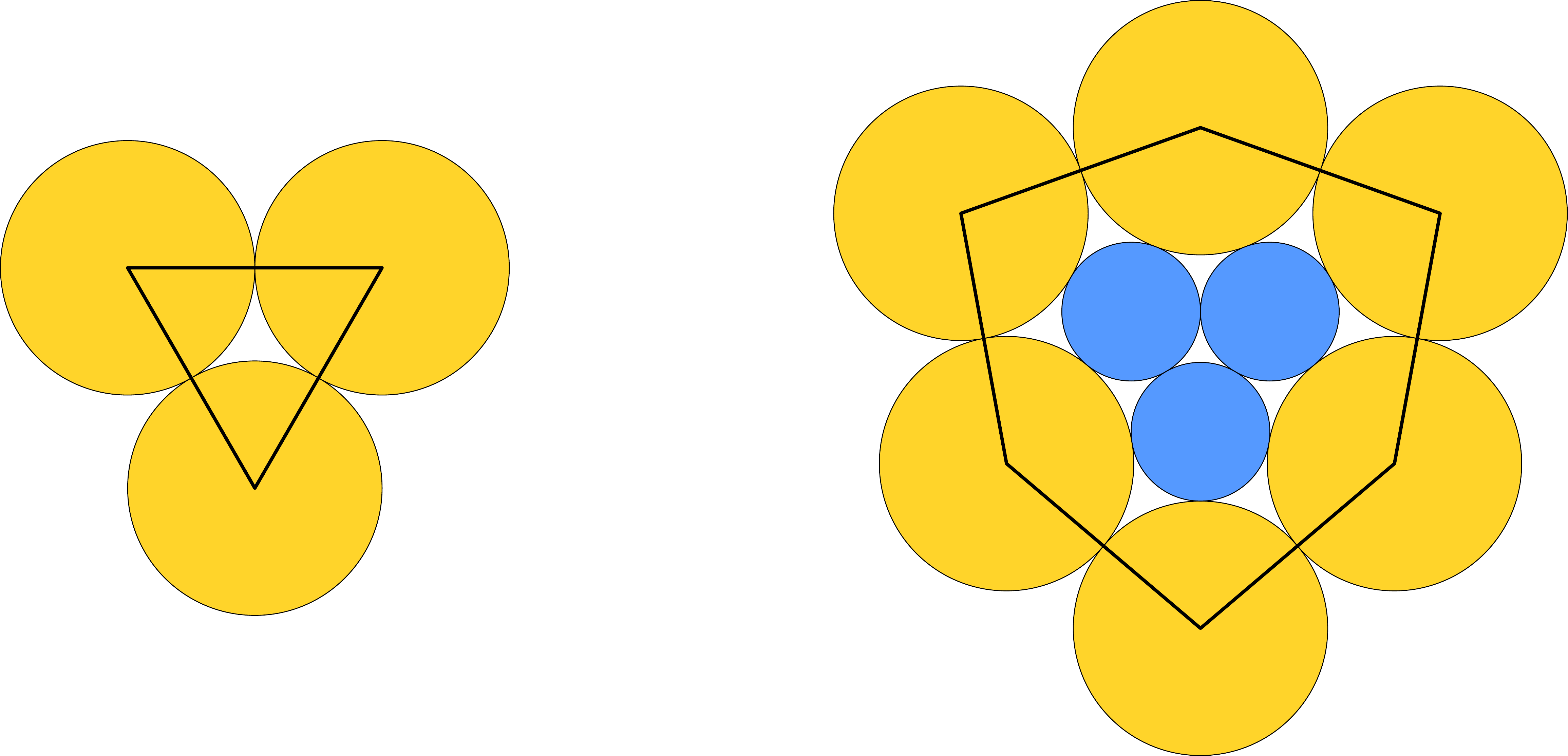}
\caption{
Two disks whose triangulated packings can be seen as shield tilings.
}
\label{fig:shield_packing}
\end{figure}

Although some shield tiles can be found in the literature (see e.g. \cite{Gah09}), the classification established here is, to our best knowledge, new.
Namely, we introduce in Section~\ref{sec:shield_tilings} two specific classes of shield tilings, called shield line tilings and shield triangle tilings, and prove in Section \ref{sec:main} that, for $\alpha\neq\tfrac{\pi}{2}$, this covers all possible shield tilings:

\begin{theorem}
\label{th:main}
For $\alpha\neq\tfrac{\pi}{2}$, every shield tiling is either a shield line tiling or a shield triangle tiling.
\end{theorem}

For $\alpha=\tfrac{\pi}{2}$ there turn out to be many more shield tilings, and a human-readable classification seems difficult to achieve.
This is discusses in Section~\ref{sec:right}.


\section{Shield line tilings and shield triangle tilings}
\label{sec:shield_tilings}

A {\em shield line} is an infinite stripe of shield tiles, aligned along one of their symetry axis, each intersecting the next one in a vertex that is also shared by two triangles; it comes in two {\em orientations} (Fig.~\ref{fig:shield_lines}).
Shield lines can be stacked one on top of the other to yield a shield tiling called a {\em shield line tiling}.
Varying the orientations yields uncountably many different tilings, all periodic in the direction of the shield lines (Fig.~\ref{fig:shield_line_tilings}).

\begin{figure}[hbt]
\centering
\includegraphics[width=\textwidth]{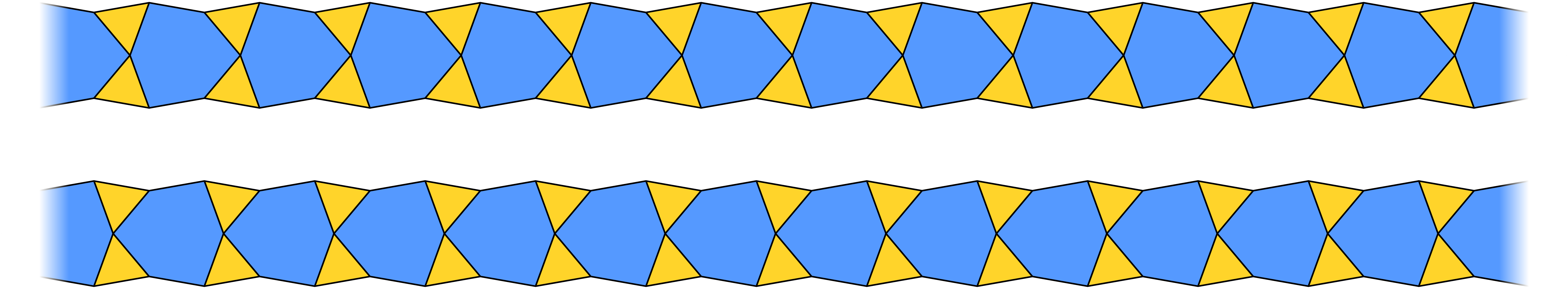}
\caption{Two parallel shield lines with opposite orientations.}
\label{fig:shield_lines}
\end{figure}

\begin{figure}[hbt]
\centering
\includegraphics[width=\textwidth]{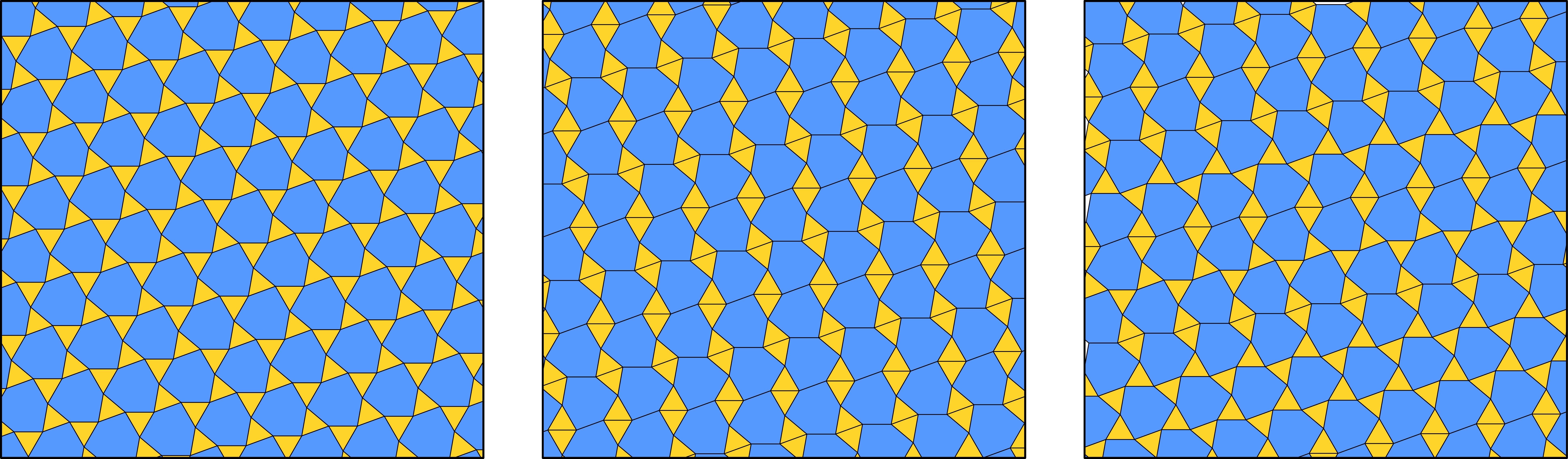}
\caption{
Shield line tilings with uniform, alternating and random orientations.
}
\label{fig:shield_line_tilings}
\end{figure}

A {\em shield triangle of order $k$} is a set of identically-oriented shield tiles centered on a triangular pattern of size $k$ of a triangular grid (whose size is minimal so that the shield are interior disjoint) with triangles completing the pattern as depicted in Fig.~\ref{fig:shield_triangles}.
Such a triangle tile the plane as the triangular grid: this yields what we call a {\em shield triangle tiling of order $k$}.
We also define the shield triangle tiling of {\em infinite order}, that it is obtained as a limit when $k$ tends to infinity.
Fig.~\ref{fig:shield_triangle_tilings} illustrates this.

\begin{figure}[hbt]
\centering
\includegraphics[width=\textwidth]{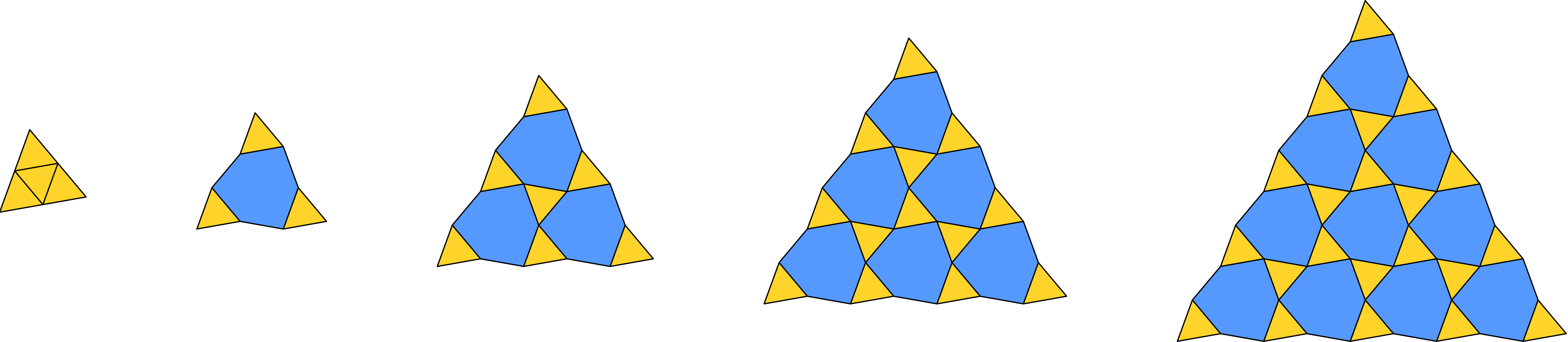}
\caption{Shield triangles of order $0$ through $4$.}
\label{fig:shield_triangles}
\end{figure}

\begin{figure}[hbt]
\centering
\includegraphics[width=\textwidth]{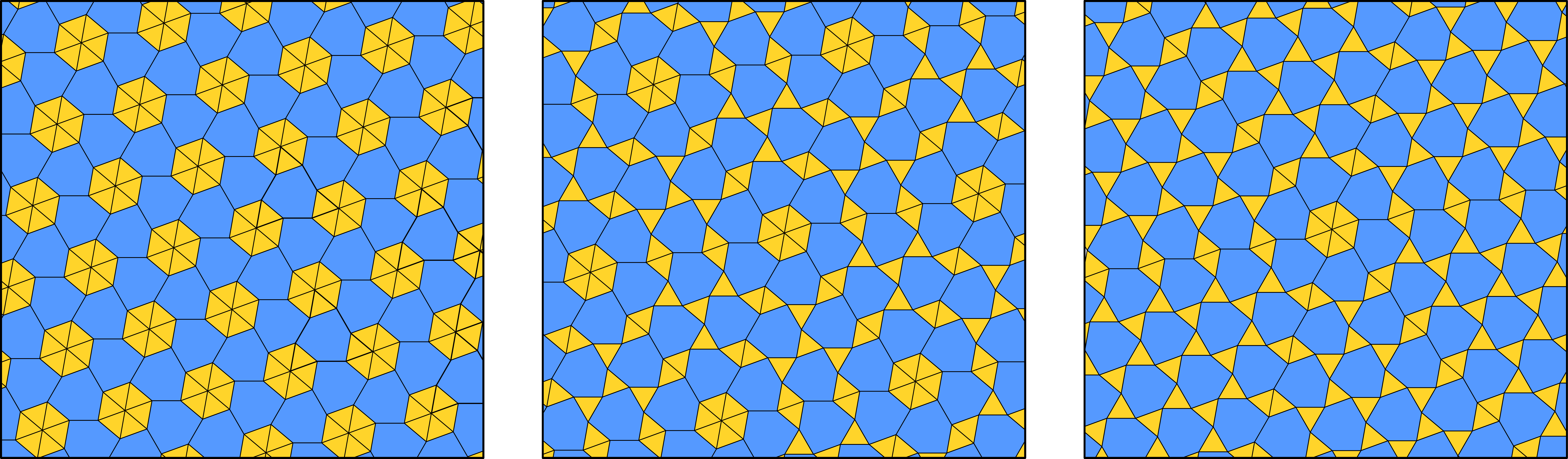}
\caption{Shield triangle tilings of order $1$, $3$ and infinite.}
\label{fig:shield_triangle_tilings}
\end{figure}

\section{Proof of Theorem~\ref{th:main}}
\label{sec:main}

\begin{lemma}
\label{lem:atlas}
For $\alpha\neq \tfrac{\pi}{2}$, the only ways shields and triangles can fit around a vertex are, up to isometry, the three ones depicted in Fig.~\ref{fig:vertex_atlas}.
\end{lemma}

\begin{figure}[hbt]
\centering
\includegraphics[width=\textwidth]{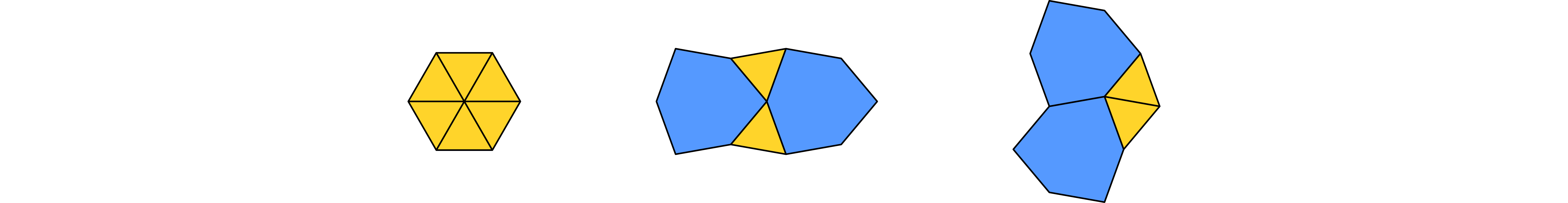}
\caption{The vertex configurations {\em hex}, {\em bowtie} and {\em fault}.}
\label{fig:vertex_atlas}
\end{figure}

\begin{proof}
Every vertex configuration yields natural numbers $p$, $q$ and $r$ such that
$$
p\alpha+q\beta+r\frac{\pi}{3}=2\pi.
$$
Since $\beta=\tfrac{4\pi}{3}-\alpha$ we can rewrite this equation
$$
(p-q)\alpha=2\pi-(4q+r)\frac{\pi}{3}.
$$
If $p=q$, then $4q+r=6$.
For $q=0$ this yields $r=6$: this corresponds to a configuration of six triangles called {\em hex}.
For $q=1$ this yields $r=2$: this corresponds to either to a configuration which alternates two triangles and two shields called {\em bowtie} or to a configuration with two neighbor triangles and two neighbor shield called {\em default}.
For $q\geq 2$ this yields $r<0$, which is impossible.

If $p\neq q$, then the equation can be satisfied only for specific values of $\alpha$.
One has $r\leq 6$, and $\alpha>\tfrac{\pi}{3}$ yields $p<6$ while $\beta>\tfrac{2\pi}{3}$ yields $q<3$.
There are thus finitely many triple $(p,q,r)$ to check.
An exhaustive check shows that, for $\alpha\neq\tfrac{\pi}{2}$, the only cases are those depicted in Fig.~\ref{fig:vertex_atlas_nongeneric}.
None of these exceptional configurations can however appear in a shield tiling because each one yields a vertex where two shields meet in their $\beta$ angle that cannot be completed.
\end{proof}

\begin{figure}[hbt]
\centering
\includegraphics[width=\textwidth]{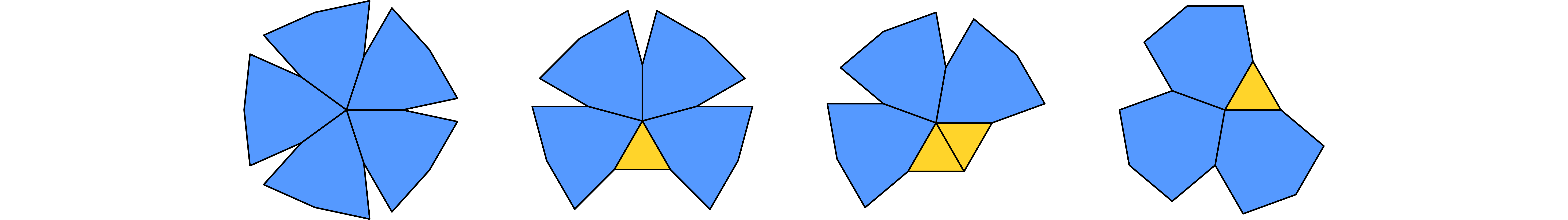}
\caption{
Exceptional vertex configurations for $\alpha=\tfrac{2\pi}{5},\tfrac{5\pi}{12},\tfrac{4\pi}{9},\tfrac{5\pi}{9}$.
}
\label{fig:vertex_atlas_nongeneric}
\end{figure}

\begin{lemma}
\label{lem:nohex}
A shield tiling with no hex is a shield line tiling.
\end{lemma}

\begin{proof}
The text refers to the figure above it.

\noindent\includegraphics[width=\textwidth]{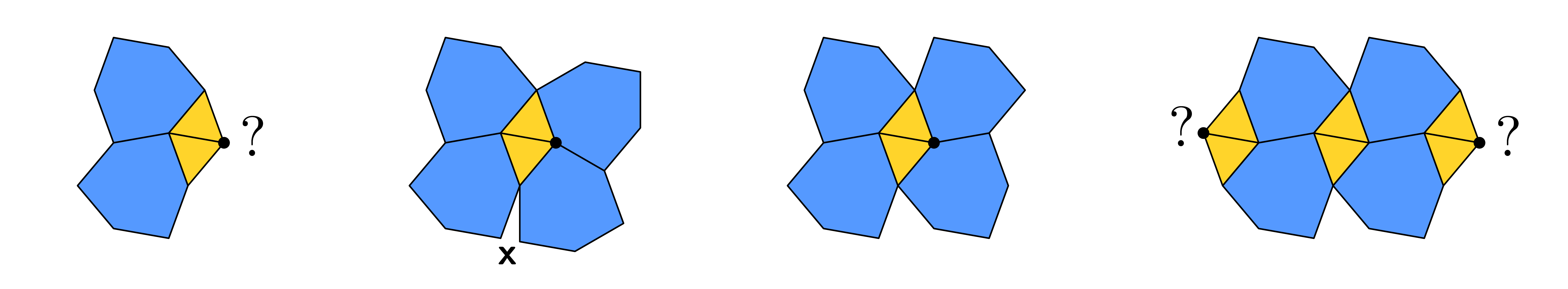}
Assume there is a fault in the tiling (left).
Its two adjacent triangles cannot belong to a hex.
They thus belong to two faults.
There are two ways to place these faults (center), but only one can appear in a tiling (the other has a thin angle that cannot be completed).
A vertex incident to two adjacent shield can only be completed by a fault, so we are brought back to the same question for new pairs of adjacent triangles (right).\\

\noindent\includegraphics[width=\textwidth]{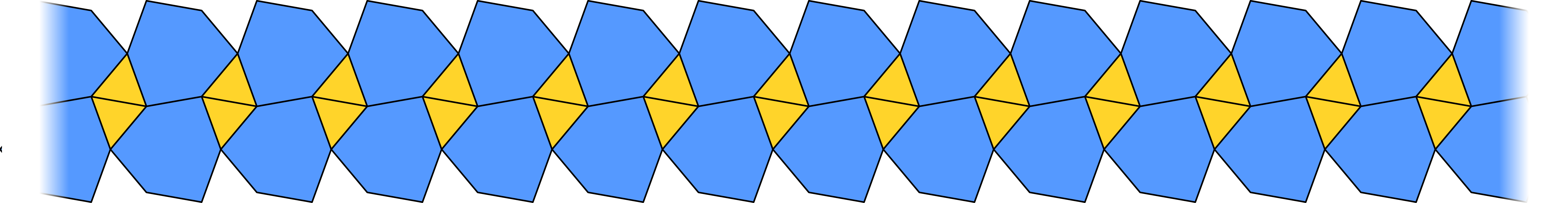}
Repeating the above argument show that the initial fault must belong to a so-called {\em fault line}, that is, the intersection of two adjacent shield lines with opposite orientations.

\noindent\includegraphics[width=\textwidth]{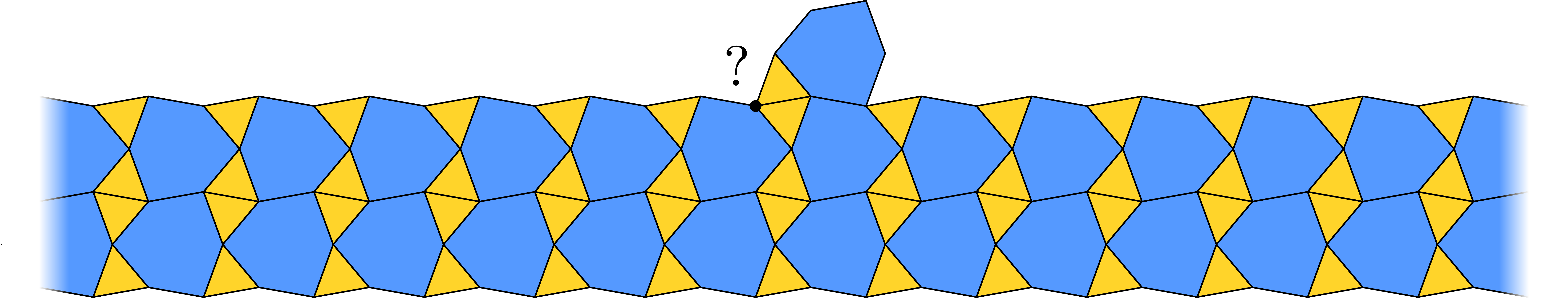}
Vertices with an open angle of $\tfrac{\pi}{3}$ can only be completed by triangles.
Add these triangles.
If any of them is involved in a fault, then the same argument as above applies and yields a new fault line.\\

\noindent\includegraphics[width=\textwidth]{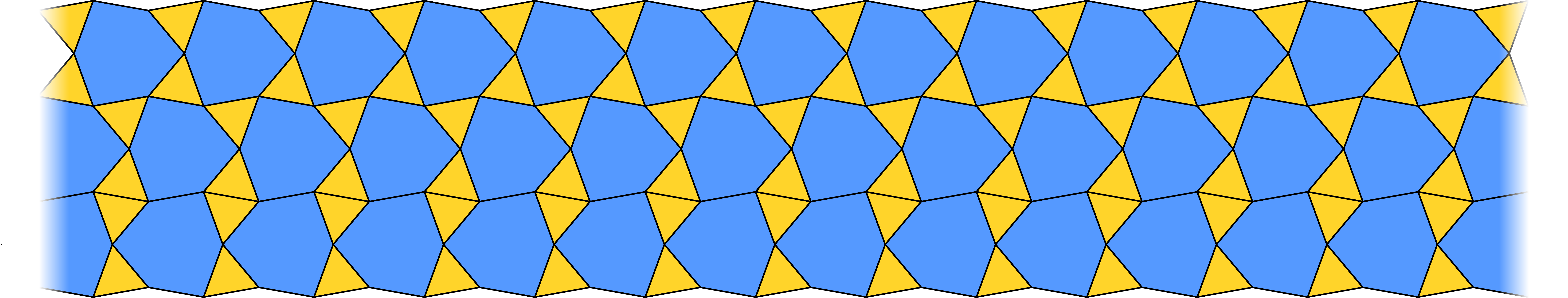}
Otherwise, that is, if all the newly added triangles are involved in a bowtie, then this yields a new shield line oriented as the previous one.\\
Iterating this argument shows that a tiling with a fault but no hex is made of stacked shield lines, hence is a shield line tiling as claimed.
If there is no fault at all (nor hex), then there are only bowties and the only possible tiling is the shield line tiling with uniform orientation (Fig.~\ref{fig:shield_line_tilings}, left).
\end{proof}

\begin{lemma}
\label{lem:hex}
A shield tiling with a hex is a shield triangle tiling.
\end{lemma}

\begin{proof}
If there are only hex, then it is the shield tiling of order $0$.\\

\noindent\includegraphics[width=\textwidth]{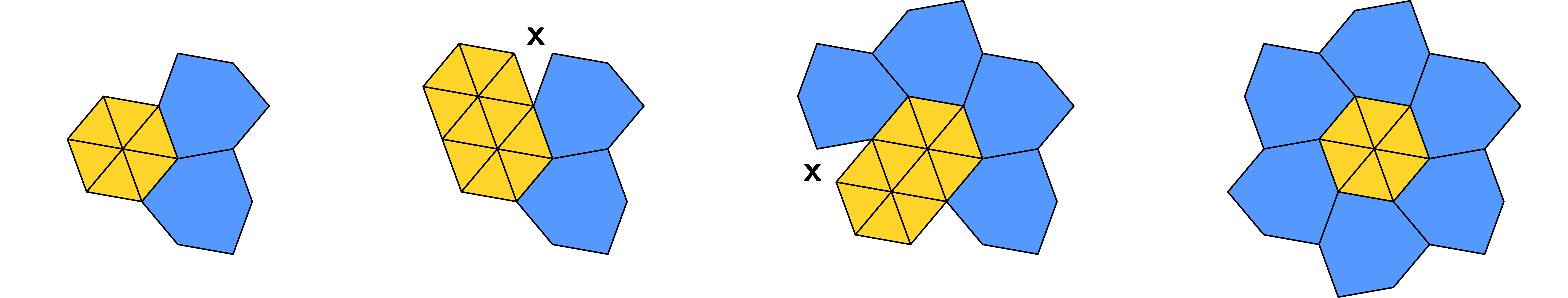}
Otherwise, consider a hex with a neighbor vertex which is not a hex.
This neighbor is necessarily a fault because it is shared by two triangles of the hex (left).
A short case study then shows that the other neighbors cannot be hex, thus are also fault (center).
Hence, any hex is surrounded by shield tiles (right).\\

\noindent\includegraphics[width=\textwidth]{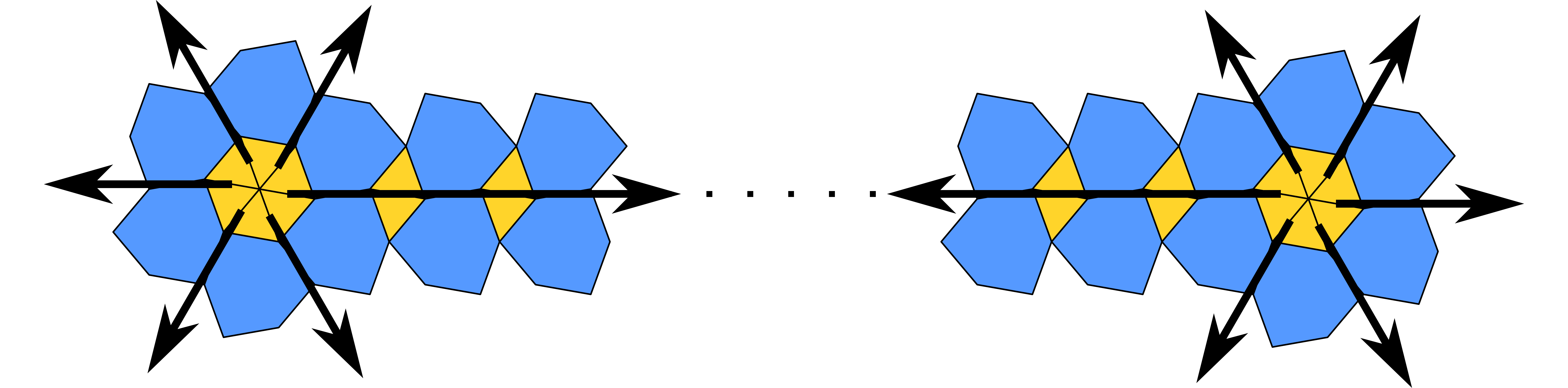}
Now, the same argument as in Lemma~\ref{lem:nohex} shows that six fault lines originate from every such hex.
Each fault line either goes on forever or eventually meets another hex.\\

\noindent\includegraphics[width=\textwidth]{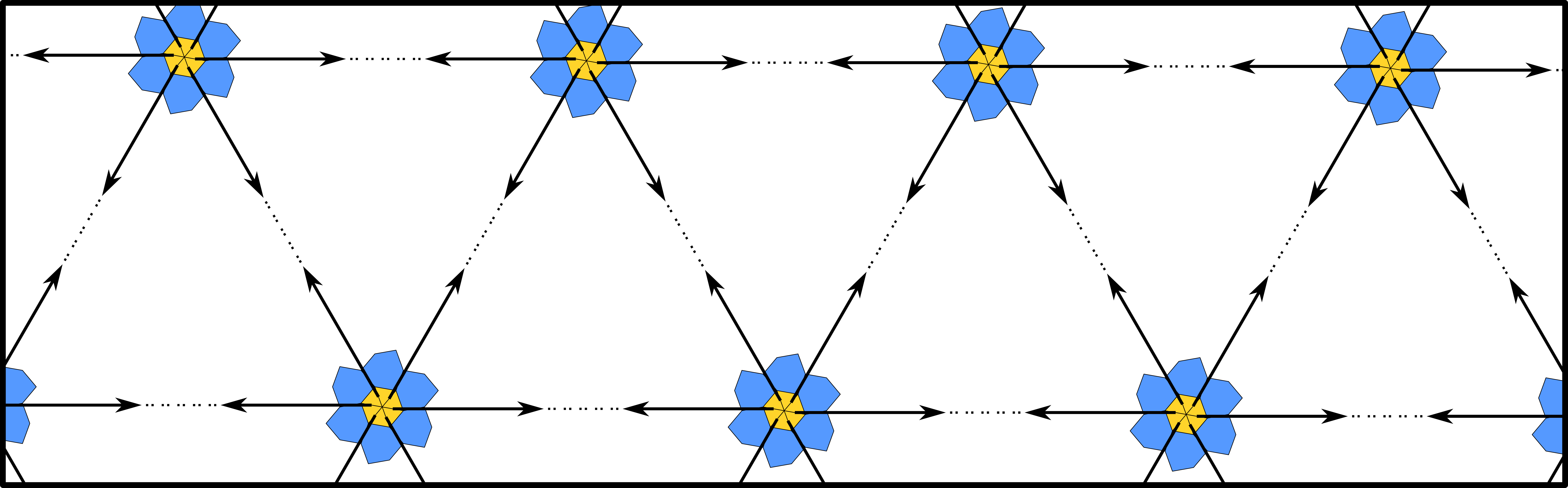}
We have seen in a Lemma~\ref{lem:nohex} that a fault vertex determines the direction of the fault line that contains it.
No vertex can thus be shared by two fault lines with different directions.
In other words, two fault lines cannot cross each other: they have to meet in a hex.
This forces the hex vertices to be arranged on the vertices of a triangular grid as depicted above.

Last, no fault line can originate in or pass through a triangle of this grid (otherwise it would have to cross a fault line on the boundary of the triangle).
There are thus neither fault nor hex vertices inside these triangles: they have all to be filled by bowties.
This yields a $T_k$ (or $T_\infty$ if there is only one hex).
\end{proof}

\noindent Theorem~\ref{th:main} directly follows from Lemmas~\ref{lem:nohex} and \ref{lem:hex}.

\section{Right shields}
\label{sec:right}

For $\alpha=\tfrac{\pi}{2}$, the shield is said to be {\em right}.
The case study of Lemma~\ref{lem:atlas} then yields three exceptional vertex configurations (Fig.~\ref{fig:vertex_atlas_right}).

\begin{figure}[hbt]
\centering
\includegraphics[width=\textwidth]{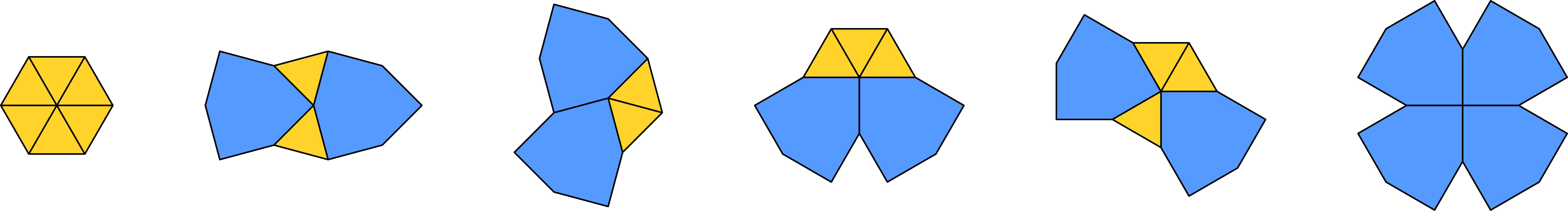}
\caption{Three generic and three exceptional vertex configurations.}
\label{fig:vertex_atlas_right}
\end{figure}

These three exceptional configurations cannot be ruled out as those in Fig.~\ref{fig:vertex_atlas_nongeneric}.
Indeed, consider a packing of regular dodecagons on the triangular grid with holes filled by triangles: a case study shows that every dodecagon can be filled in exactly three different ways (Fig.~\ref{fig:right_shield_tilings}, left).
In particular, this allows tilings where all the $6$ vertex configurations appear (Fig.~\ref{fig:right_shield_tilings}, center).

\begin{figure}[hbt]
\centering
\includegraphics[width=\textwidth]{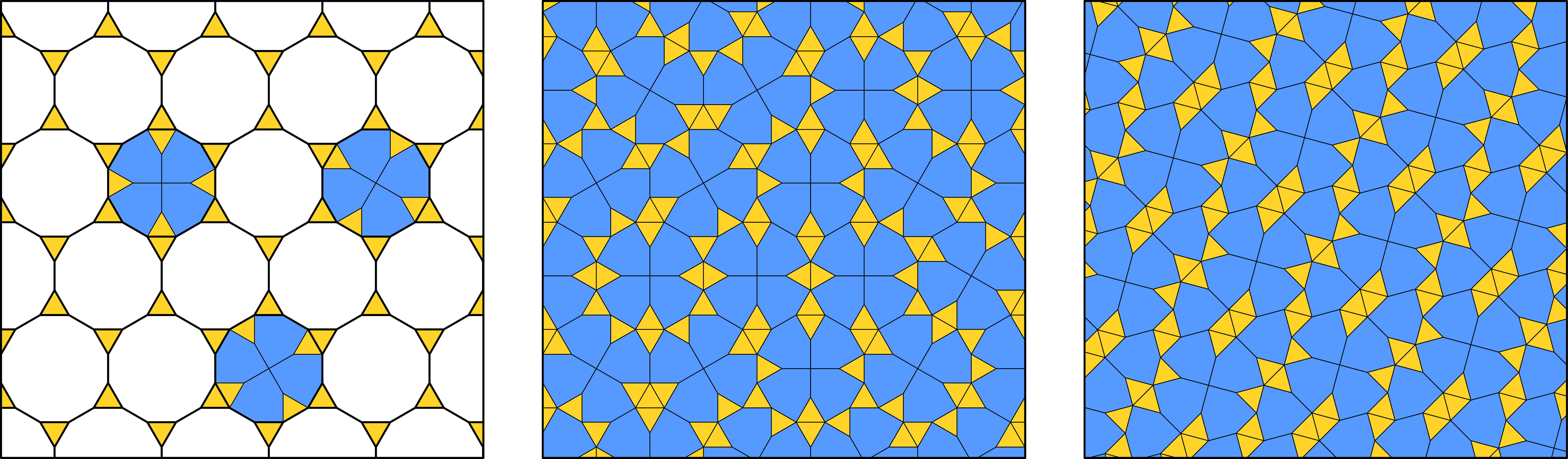}
\caption{
A packing of regular dodecagons and the three different ways to fill each one (left).
A random filling of these dodecagons that contains every possible vertex configuration (center).
Another right shield tiling which cannot be obtained in this way (right).
}
\label{fig:right_shield_tilings}
\end{figure}

These dodecagon-based shield tilings also show that the number $P_n$ of different patterns that can be obtained by taking all the tiles within distance $n$ from some vertex of some shield tiling grows (at least) exponentially fast in $n^2$.
In particular, the {\em entropy} of right shield tilings, defined as the limit superior of $\log(P_n)/n^2$, is positive.
In contrast, the entropy of generic shield tilings is zero because $P_n$ grows only exponentially in $n$.

And yet, the above dodecagon-based shield tilings do not exhaust the subject: there are still other shield tilings that cannot be obtained in this way (see e.g. Fig.~\ref{fig:right_shield_tilings}, right).
Alike the square and triangle tilings mentioned in the introduction, shield tilings may be too ``wild'' to admit a human-readable description\ldots

\bibliographystyle{alpha}
\bibliography{shield}
\end{document}